\documentclass{article}
\usepackage{amsfonts, amstext, amsmath, amsthm, amscd, amssymb, hyperref, float, longtable, enumitem, fullpage}
\usepackage{epsfig, graphics, psfrag}
\usepackage{graphicx}
\usepackage{color}
\usepackage[font=small, labelfont=bf]{caption}
\usepackage{float}

\newtheorem{definition}{Definition}[section]
\let\olddefinition\definition
\renewcommand{\definition}{\olddefinition\normalfont}
\newtheorem{remark}{Remark}[section]
\let\oldremark\remark
\renewcommand{\remark}{\oldremark\normalfont}
\newtheorem{proposition}{Proposition}[section]
\newtheorem{theorem}{Theorem}[section]
\newtheorem{lemma}{Lemma}[section]
\newtheorem{sublemma}{Sublemma}[section]
\newtheorem{corollary}{Corollary}[section]

\begin{document}
\title{Semi-Adequate Closed Braids and Volume}
\author{Adam Giambrone\\ Alma College\\ giambroneaj@alma.edu}
\date{}
\maketitle

\begin{abstract}

In this paper, we show that the volumes for a family of A-adequate closed braids can be bounded above and below in terms of the twist number, the number of braid strings, and a quantity that can be read from the combinatorics of a given closed braid diagram. We also show that the volumes for many of these closed braids can be bounded in terms of a single stable coefficient of the colored Jones polynomial, thus showing that this collection of closed braids satisfies a Coarse Volume Conjecture. By expanding to a wider family of closed braids, we also obtain volume bounds in terms of the number of positive and negative twist regions in a given closed braid diagram. Furthermore, for a family of A-adequate closed $3$-braids, we show that the volumes can be bounded in terms of the parameter $s$ from the Schreier normal form of the $3$-braid. Finally we show that, for the same family of A-adequate closed 3-braids, the parameters $k$ and $s$ from the Schreier normal form can actually be read off of the original 3-braid word. 

\end{abstract}

\section{Introduction}

One of the current aims of knot theory is to strengthen the relationships among the hyperbolic volume of the link complement, the colored Jones polynomial, and data extracted from link diagrams. In a recent monograph, Futer, Kalfagianni, and Purcell (\cite{Guts}, or see \cite{Survey} for a survey of results) showed that, for sufficiently twisted negative braid closures and for certain Montesinos links, the volume of the link complement can be bounded above and below in terms of the twist number of an A-adequate link diagram. The results for Montesinos links were recently generalized by Purcell and Finlinson in \cite{Montesinos}. Similar results for alternating links were previously given by Lackenby in \cite{Lackenby}, with the lower bounds improved upon by Agol, Storm, and W. Thurston in \cite{AgolStorm} and the upper bounds improved upon by Agol and D. Thurston in the appendix of \cite{Lackenby} and more recently improved upon by Dasbach and Tsvietkova in \cite{Refined}. In a previous paper (\cite{Me}), the author showed that the volumes for a large family of A-adequate link complements can be bounded in terms of two diagrammatic quantities: the twist number and the number of certain alternating tangles in a given A-adequate diagram. 

%Let $t(D)$ denote the twist number of a link diagram $D(K)$ and let $st(D)$ denote the number of special tangles in $D(K)$. 

%The main result of \cite{Me} (which incorporates results from \cite{AgolStorm}, \cite{New}, \cite{Guts}, and \cite{Lackenby}) is stated below.  
%
%\begin{theorem}[Theorem 1.1 of \cite{Me}] Let $D(K)$ be a connected, prime, A-adequate link diagram that satisfies the two-edge loop condition and contains $t(D)\geq2$ twist regions. Then $K$ is hyperbolic and $t(D) \geq st(D)$. Furthermore, we get the following two possibilities.  
%
%\begin{enumerate}
	%\item[(1)] If $t(D)=st(D)$, then $D(K)$ is alternating and  
	%\begin{equation*}
%\frac{v_{8}}{2}\cdot\left(t(D)-2\right) \leq \mathrm{vol}(S^{3}\backslash K) < 10v_{3}\cdot(t(D)-1).
%\end{equation*}
		%\item[(2)] If $t(D)>st(D)$, then 
	%\begin{equation*}
%\frac{v_{8}}{3}\cdot\left(t(D)-st(D)\right) \leq \mathrm{vol}(S^{3}\backslash K) < 10v_{3}\cdot(t(D)-1).
%\end{equation*}
%\end{enumerate}
%\label{mainthm} 
%\end{theorem}
%
%Note that the coefficients of $t(D)$ in the upper and lower bounds differ by a multiplicative factor of $8.3102\ldots$. We would like to improve this factor by studying specific families of links. 

In this paper, we show that the volumes for a family of A-adequate closed $n$-braids can be bounded above and below in terms of the twist number $t(D)$, the number $t^{+}(D)$ of positive twist regions, the number $t^{-}(D)$ of negative twist regions, the number $n$ of braid strings, and the number $m$ of special types of state circles (called \emph{non-essential wandering circles}) that arise from a given closed braid diagram. Let $v_{8}=3.6638\ldots$ denote the volume of a regular ideal octahedron and let $v_{3}=1.0149\ldots$ denote the volume of a regular ideal tetrahedron. The main results of this paper, where the words ``certain'' and ``more general'' will be made precise later, are stated below. 

\begin{theorem} \label{introthm} For $D(K)$ a certain A-adequate closed $n$-braid diagram, the complement of $K$ satisfies the volume bounds
$$\frac{v_{8}}{2}\cdot(t(D)-2(n+m-2)) \leq v_{8}\cdot(t^{-}(D)-(n+m-2)) \leq \mathrm{vol}(S^{3}\backslash K) < 10v_{3}\cdot(t(D)-1).$$
For $D(K)$ a more general A-adequate closed $n$-braid diagram, the complement of $K$ satisfies the volume bounds
$$v_{8}\cdot(t^{-}(D)-t^{+}(D)-(n+m-2)) \leq \mathrm{vol}(S^{3}\backslash K) < 10v_{3}\cdot(t(D)-1) = 10v_{3}\cdot(t^{-}(D)+t^{+}(D)-1).$$
\end{theorem}

By restricting to a family of A-adequate closed $3$-braids, we show that the volumes can also be bounded in terms of the parameter $s$ from the Schreier normal form of the $3$-braid. It should be noted that the lower bound provided in this paper, which relies on the more recent machinery of \cite{Guts}, is often an improvement over the one given in \cite{Cusp}.  

\begin{theorem}\label{sthm}
For $D(K)$ a certain A-adequate closed $3$-braid diagram, the complement of $K$ satisfies the volume bounds
$$v_{8}\cdot(s-1) \leq \mathrm{vol}(S^{3}\backslash K) < 4v_{8}\cdot s.$$
\end{theorem}

In addition to providing diagrammatic volume bounds, we also show that, for the same family of A-adequate closed 3-braids, the parameters $k$ and $s$ from the Schreier normal form can actually be read off of the original 3-braid diagram. 

The volumes for many families of link complements have also been expressed in terms of coefficients of the colored Jones polynomial (\cite{Volumish}, \cite{Refined}, \cite{Guts}, \cite{Filling}, \cite{Symmetric}, \cite{Cusp}, \cite{Stoimenow}). Denote the \emph{$j^{th}$ colored Jones polynomial} of a link $K$ by
$$J_{K}^{j}(t)=\alpha_{j}t^{m_{j}}+\beta_{j}t^{m_{j}-1}+\cdots+\beta_{j}'t^{r_{j}+1}+\alpha_{j}'t^{r_{j}},$$
\noindent where $j \in \mathbb{N}$ and where the degree of each monomial summand decreases from left to right. We will show that, for fixed $n$, the volumes for many of the closed $n$-braids considered in this paper can be bounded in terms of the stable penultimate coefficient $\beta_{K}':=\beta_{j}'$ (where $j \geq 2$) of the colored Jones polynomial. A result of this nature shows that the given collection of closed $n$-braids satisfies a Coarse Volume Conjecture (\cite{Guts}, Section 10.4). This result, which is stated below, can be viewed as a corollary of Theorem~\ref{introthm}.  

\begin{corollary}\label{newcor} 
For $D(K)$ a certain A-adequate closed $n$-braid diagram, the complement of $K$ satisfies the volume bounds
$$v_{8}\cdot(\left|\beta_{K}'\right|-1) \leq \mathrm{vol}(S^{3}\backslash K) < 20v_{3}\cdot\left(\left|\beta_{K}'\right|+n+m-\dfrac{7}{2}\right).$$
\end{corollary}

\section*{Acknowledgments}
I would like to thank Efstratia Kalfagianni for helpful conversations that led to this project. I would also like to thank the anonymous referee for helpful comments that led to an improvement in the quality of exposition in this paper.

\section{Volume bounds for A-adequate closed braids}

\subsection{Preliminaries}

\begin{figure}
	\centering
		\def\svgwidth{2.5in}
		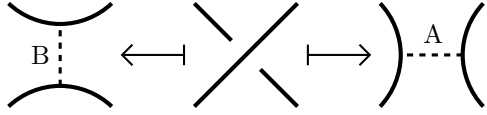
	\caption{A crossing neighborhood of a link diagram, along with its A-resolution and B-resolution.}
	\label{resolutions}
\end{figure}

Let $D(K) \subseteq S^2$ denote a diagram of a link $K \subseteq S^3$. To smooth a crossing of the link diagram $D(K)$, we may either \emph{A-resolve} or \emph{B-resolve} this crossing according to Figure~\ref{resolutions}. By A-resolving each crossing of $D(K)$ we form the \emph{all-A state} of $D(K)$, which is denoted by $H_{A}$ and consists of a disjoint collection of \emph{all-A circles} and a disjoint collection of dotted line segments, called \emph{A-segments}, that are used record the locations of crossing resolutions. We will adopt the convention throughout this paper that any unlabeled segments are assumed to be A-segments. We call a link diagram $D(K)$ \emph{A-adequate} if $H_{A}$ does not contain any A-segments that join an all-A circle to itself, and we call a link $K$ \emph{A-adequate} if it has a diagram that is A-adequate. While we will focus exclusively on A-adequate links, our results can easily be extended to semi-adequate links by reflecting the link diagram $D(K)$ and obtaining the corresponding results for B-adequate links. 

From $H_{A}$ we may form the \emph{all-A graph}, denoted $\mathbb{G}_{A}$, by contracting the all-A circles to vertices and reinterpreting the A-segments as edges. From this graph we can form the \emph{reduced all-A graph}, denoted $\mathbb{G}_{A}'$, by replacing parallel edges with a single edge. For an example of a diagram $D(K)$, its all-A resolution $H_{A}$, its all-A graph $\mathbb{G}_{A}$, and its reduced all-A graph $\mathbb{G}_{A}'$, see Figure~\ref{figure8}. Let $v(G)$ and $e(G)$ denote the number of vertices and edges, respectively, in a graph $G$. Let $-\chi(G)=e(G)-v(G)$ denote the negative Euler characteristic of $G$. 

\begin{figure}
	\centering
		\def\svgwidth{3.5in}
		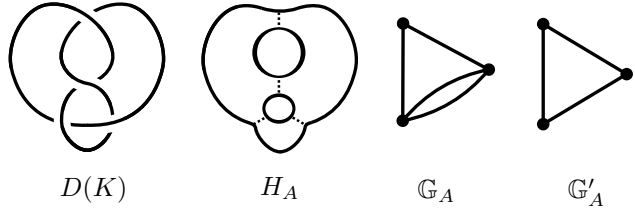
	\caption{A link diagram $D(K)$, its all-A resolution $H_{A}$, its all-A graph $\mathbb{G}_{A}$, and its reduced all-A graph $\mathbb{G}_{A}'$.}
	\label{figure8}
\end{figure}

%\begin{remark}
%\label{circleremark}
%Note that $v(\mathbb{G}_{A}')$ is the same as the number of all-A circles in $H_{A}$ and that $e(\mathbb{G}_{A})$ is the same as the number of A-segments in $H_{A}$. From a graphical perspective, A-adequacy of $D(K)$ can equivalently be defined by the condition that $\mathbb{G}_{A}$ contains no one-edge loops that connect a vertex to itself.   
%\end{remark}

\begin{figure}
	\centering
		\def\svgwidth{200pt}
		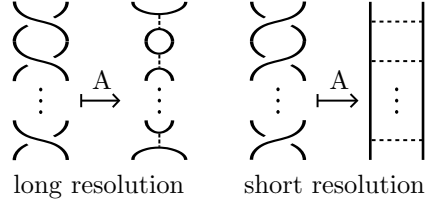
	\caption{The long and short resolutions of a twist region of $D(K)$.}
	\label{longshort}
\end{figure}

\begin{definition} 
Define a \emph{twist region} of $D(K)$ to be a longest possible string of bigons in the projection graph of $D(K)$. Denote the number of twist regions in $D(K)$ by $t(D)$ and call $t(D)$ the \emph{twist number} of $D(K)$. Note that it is possible for a twist region to consist of a single crossing of $D(K)$. If a given twist region contains two or more crossings, then the A-resolution of a left-handed twist region will be called a \emph{long resolution} and the A-resolution of a right-handed twist region will be called a \emph{short resolution}. See Figure~\ref{longshort} for depictions of these resolutions. We will call a twist region \emph{long} if its A-resolution is long and \emph{short} if its A-resolution is short. 
\end{definition}

\begin{definition}
A link diagram $D(K)$ satisfies the \emph{two-edge loop condition (TELC)} if, whenever two all-A circles share a pair of A-segments, these segments correspond to crossings from the same short twist region of $D(K)$. (See the right side of Figure~\ref{longshort}.)
\end{definition}

Recall that $v_{8}=3.6638\ldots$ denotes the volume of a regular ideal octahedron and $v_{3}=1.0149\ldots$ denotes the volume of a regular ideal tetrahedron. By combining results from \cite{New}, \cite{Guts}, and \cite{Lackenby}, we get the following key result.

\begin{theorem}[Corollary 1.4 of \cite{New}, Theorem from Appendix of \cite{Lackenby}]
Let $D(K)$ be a connected, prime, A-adequate link diagram that satisfies the TELC and contains $t(D) \geq 2$ twist regions. Then $K$ is hyperbolic and
\label{Cor}
\begin{equation*}
-v_{8}\cdot\chi(\mathbb{G}_{A}') \leq \mathrm{vol}(S^{3}\backslash K) < 10v_{3}\cdot(t(D)-1).
\end{equation*}
\end{theorem}

\subsection{Braids and A-adequacy} 

% INTRODUCE SECTION!!!

Recall that the \emph{n-braid group} has Artin presentation

$$\displaystyle B_{n} = \left\langle \sigma_{1}, \cdots, \sigma_{n-1}\ |\ \sigma_{i}\sigma_{j}=\sigma_{j}\sigma_{i}\ \mathrm{and}\ \sigma_{i}\sigma_{i+1}\sigma_{i} = \sigma_{i+1}\sigma_{i}\sigma_{i+1}\right\rangle,$$ 

\noindent where the generators $\sigma_{i}$ are depicted in Figure~\ref{braidgen}, where $\left|i-j\right|\geq2$ and $1 \leq i < j \leq n-1$ in the first relation (sometimes called \emph{far commutativity}), and where $1\leq i\leq n-2$ in the second relation. As a special case, the \emph{3-braid group} has Artin presentation

$$B_{3}=\left\langle \sigma_{1}, \sigma_{2}\ |\ \sigma_{1}\sigma_{2}\sigma_{1}=\sigma_{2}\sigma_{1}\sigma_{2}\right\rangle.$$ 

\begin{figure}
	\centering
	\def\svgwidth{3in}
		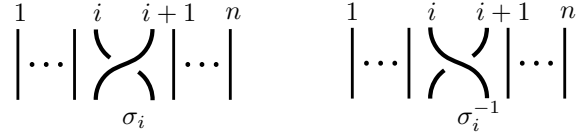
	\caption{The braid group generators $\sigma_{i}$ and $\sigma_{i}^{-1}$ of $B_{n}$, where $1 \leq i \leq n-1$.}
	\label{braidgen}
\end{figure}

\begin{definition}
Call a subword $\gamma$ of a braid word $\beta \in B_{n}$ \emph{(cyclically) induced} if the letters of $\gamma$ are all (cyclically) adjacent and appear in the same order as in the full braid word $\beta$. As an example, given the braid word 

$$\beta=\sigma_{1}^{3}\sigma_{2}^{-3}\sigma_{1}^{2}\sigma_{3}^{-2}\sigma_{2}\sigma_{3}=\sigma_{1}^{3}\sigma_{2}^{-1}(\sigma_{2}^{-2}\sigma_{1}^{2}\sigma_{3}^{-1})\sigma_{3}^{-1}\sigma_{2}\sigma_{3},$$ 

\noindent the subword $\gamma=\sigma_{2}^{-2}\sigma_{1}^{2}\sigma_{3}^{-1}$ is an induced subword of $\beta$. 
\end{definition} 

Because conjugate braids have isotopic braid closures, we will usually work within conjugacy classes of braids. Furthermore, note that cyclic permutation is a special case of conjugation.

\begin{definition} A braid $\beta=\sigma_{m_{1}}^{r_{1}}\cdots\sigma_{m_{l}}^{r_{l}} \in B_{n}$ is called \emph{cyclically reduced into syllables} if the following hold:

\begin{enumerate}
\item[(1)] $r_{i}\neq 0$ for all $i$. 
\item[(2)] there are no occurrences of induced subwords of the form $\sigma_{i}\sigma_{i}^{-1}$ or $\sigma_{i}^{-1}\sigma_{i}$ for any $i$ (looking up to cyclic permutation).
\item[(3)] $m_{i}\neq m_{i+1}$ for all $i$ (modulo $l$). 
\end{enumerate}
\end{definition}

Recall that a braid $\beta=\sigma_{m_{1}}^{r_{1}}\cdots\sigma_{m_{l}}^{r_{l}} \in B_{n}$ is called \emph{positive} if all of the exponents $r_{i}$ are positive and \emph{negative} if all of the exponents $r_{i}$ are negative. We define a syllable $\sigma_{m_{i}}^{r_{i}}$ of the braid word to be \emph{positive} if $r_{i}>0$ and \emph{negative} if $r_{i}<0$. We now present Stoimenow's (\cite{Stoimenow}) classification of A-adequate closed 3-braid diagrams. 

\begin{proposition}[\cite{Stoimenow}, Lemma 6.1] 

Let $D(K)=\widehat{\beta}$ denote the closure of a 3-braid 

$$\beta=\sigma_{i}^{r_{1}}\sigma_{j}^{r_{2}}\cdots\sigma_{i}^{r_{2l-1}}\sigma_{j}^{r_{2l}} \in B_{3},$$ 

\noindent where $\left\{i,j\right\}=\left\{1,2\right\}$, where $l \geq 2$, and where $\beta$ has been cyclically reduced into syllables. Then $D(K)$ is A-adequate if and only if either 
\begin{enumerate}\label{adeq3braid}
\item[(1)] $\beta$ is positive, or
\item[(2)] $\beta$ does not contain $\sigma_{1}^{-1}\sigma_{2}^{-1}\sigma_{1}^{-1}=\sigma_{2}^{-1}\sigma_{1}^{-1}\sigma_{2}^{-1}$ as a cyclically induced subword and $\beta$ also has the property that all positive syllables in the braid word are cyclically neighbored on both sides by negative syllables.   
\end{enumerate}
\end{proposition}

\subsection{A lemma for closed braids} 

The main goal of this section is to produce a family of closed $n$-braids that satisfies the hypotheses of Theorem~\ref{Cor}. We will work within this family of braids throughout the rest of this paper. We begin with necessary definitions.  

\begin{definition}
Call two induced subwords $\gamma_{1}$ and $\gamma_{2}$ of a braid word $\beta \in B_{n}$ \emph{disjoint} if the subwords share no common letters when they are viewed as part of $\beta$. As an example, given the braid word
$$\beta=\sigma_{1}^{3}\sigma_{2}^{-3}\sigma_{1}^{2}\sigma_{3}^{-2}\sigma_{2}\sigma_{3}=\sigma_{1}^{2}(\sigma_{1}\sigma_{2}^{-3}\sigma_{1})\sigma_{1}(\sigma_{3}^{-2}\sigma_{2})\sigma_{3},$$ 
\noindent the subwords $\gamma_{1}=\sigma_{1}\sigma_{2}^{-3}\sigma_{1}$ and $\gamma_{2}=\sigma_{3}^{-2}\sigma_{2}$ are disjoint induced subwords of $\beta$. 
\end{definition}

\begin{definition}
Call a subword $\gamma$ of a braid word $\beta \in B_{n}$ \emph{complete} if it contains, at some point, each of the generators $\sigma_{1}, \ldots, \sigma_{n-1}$ of $B_{n}$. It is important to note that the generators $\sigma_{1}, \ldots, \sigma_{n-1}$ need not occur in any particular order and that repetition of some or all of these generators is allowed.
\end{definition}

Because the majority of the braids considered in this paper will satisfy the same set of assumptions, we make the following definition. 

\begin{definition}
Call a braid $\beta \in B_{n}$ \emph{nice} if $\beta$ is cyclically reduced into syllables and $\beta$ contains two disjoint induced complete subwords.
\end{definition}

Note that, in the case when $n=3$, the condition that the 3-braid $\beta=\sigma_{i}^{r_{1}}\sigma_{j}^{r_{2}}\cdots\sigma_{i}^{r_{2l-1}}\sigma_{j}^{r_{2l}} \in B_{3}$ is nice is equivalent to the condition that $\sigma_{1}$ and $\sigma_{2}$ each occur at least twice nontrivially in the braid word $\beta$ (which is equivalent to the condition that $l \geq 2$). 

We now define, in the following Main Lemma, the family of closed $n$-braids that we will consider for the remainder of this paper. 

\begin{lemma}[The Main Lemma] 
\label{setup}
Let $D(K)=\widehat{\beta}$ denote the closure of a nice $n$-braid $\beta=\sigma_{m_{1}}^{r_{1}}\cdots\sigma_{m_{l}}^{r_{l}} \in B_{n}$. View $D(K)$ as lying in an annular region of the plane and assume that $\beta$ satisfies the conditions
\begin{enumerate}   
\item[(1)] all negative exponents $r_{i}<0$ in $\beta$ satisfy the stronger requirement that $r_{i}\leq-3$;
\item[(2a)] when $r_{i}>0$ and $m_{i}=1$, we have that $r_{i-1} \leq -3$, that $r_{i+1} \leq -3$, and that $m_{i-1}=m_{i+1}=2$; 
\item[(2b)] when $r_{i}>0$ and $2 \leq m_{i} \leq n-2$, we have that $r_{i-2} \leq -3$, that $r_{i-1} \leq -3$, that $r_{i+1} \leq -3$, that $r_{i+2} \leq -3$, and that $\{m_{i-2}, m_{i-1}\}=\{m_{i+1}, m_{i+2}\}=\{m_{i}-1, m_{i}+1\}$; and
\item[(2c)] when $r_{i}>0$ and $m_{i}=n-1$, we have that $r_{i-1} \leq -3$, that $r_{i+1} \leq -3$, and that $m_{i-1}=m_{i+1}=n-2$. 
\end{enumerate}
Then we may categorize the all-A circles of $H_{A}$ into the following types:
\begin{itemize}
\item \emph{small inner circles} that arise from negative exponents $r_{i}\leq-3$ in the braid word $\beta$
\item \emph{medium inner circles} that arise from cyclically isolated positive syllables in the braid word $\beta$ 
\item \emph{essential wandering circles} that are essential in the annulus and have wandering that arises from adjacent negative syllables in adjacent generators in the braid word $\beta$ 
\item \emph{non-essential wandering circles} that are non-essential (contractible) in the annulus and have wandering that arises from adjacent negative syllables in adjacent generators in the braid word $\beta$  
\item \emph{nonwandering circles} that arise from the case when all syllables in the generator $\sigma_{1}$ are positive or the case when all syllables in the generator $\sigma_{n-1}$ are positive.
\end{itemize}
\noindent Furthermore, we have that $K$ is a hyperbolic link and that $D(K)$ is a connected, prime, A-adequate link diagram that satisfies the TELC and contains $t(D)~\geq~2(n-1)$ twist regions.    
\end{lemma}

\begin{remark} \label{cond}
Condition~(2a) of the above theorem requires that positive syllables in the generator $\sigma_{1}$ are cyclically neighbored on both sides by negative syllables in the generator $\sigma_{2}$. Similarly, Condition~(2c) of the above theorem requires that positive syllables in the generator $\sigma_{n-1}$ are cyclically neighbored on both sides by negative syllables in the generator $\sigma_{n-2}$. Condition~(2b) of the above theorem requires that positive syllables in the generator $\sigma_{i}$, where $2 \leq i \leq n-2$, are cyclically neighbored on both sides by a pair of negative syllables in the far commuting generators $\sigma_{i-1}$ and $\sigma_{i+1}$.
\end{remark}

Note that Condition~(2a), Condition~(2b), and Condition~(2c) are in a spirit similar to the condition that positive syllables are cyclically neighbored on both sides by negative syllables in Stoimenow's classification of A-adequate closed 3-braid diagrams (Proposition~\ref{adeq3braid}). 

\begin{figure}
	\centering
		\includegraphics[width=3in]{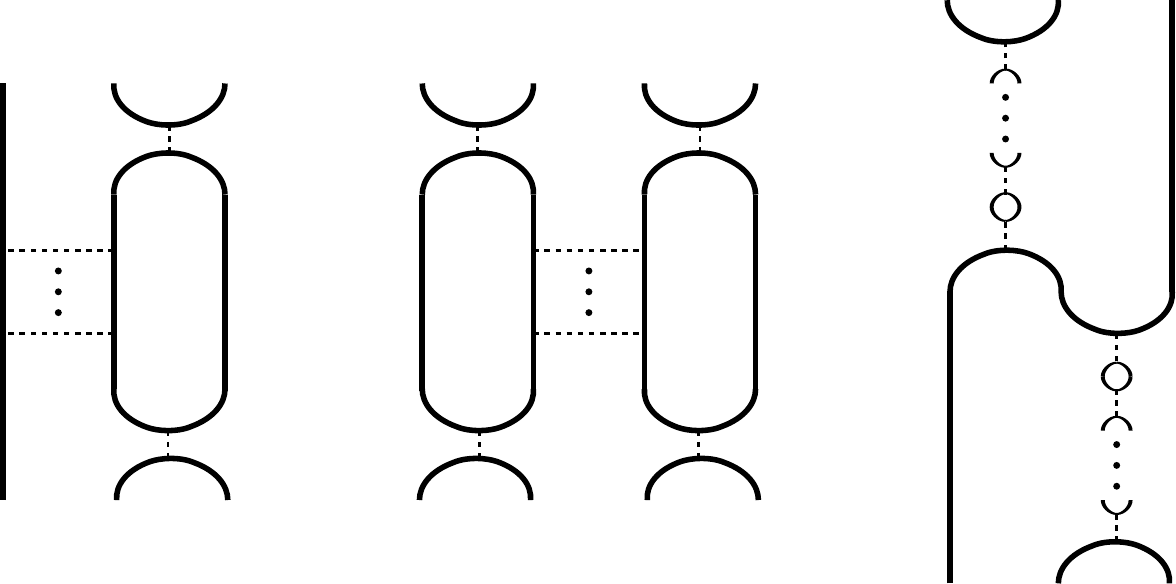}
	\caption{An example of an isolated positive syllable $\sigma_{1}^{r_{i}}$ corresponding to a single medium inner circle (left), an example of an isolated positive syllable $\sigma_{m_{i}}^{r_{i}}$, where $2 \leq m_{i} \leq n-2$, corresponding to two medium inner circles (center), and an example of a wandering circle portion coming from a negative subword $\sigma_{m_{i}}^{r_{i}}\sigma_{m_{i+1}}^{r_{i+1}}=\sigma_{{m_i}}^{r_{i}+1}(\sigma_{m_{i}}^{-1}\sigma_{m_{i+1}}^{-1})\sigma_{m_{i+1}}^{r_{i+1}+1}$ (right).}
	\label{loop}
\end{figure}

We will prove Lemma~\ref{setup} through a series of sublemmas. 

\begin{sublemma} \label{sublemma1} Let $D(K)=\widehat{\beta}$ satisfy the assumptions of Lemma~\ref{setup}. Then, as described in the statement of Lemma~\ref{setup}, we may categorize the all-A circles of $H_{A}$ into small inner circles, medium inner circles, essential wandering circles, non-essential wandering circles, and nonwandering circles.  
\end{sublemma}

\begin{proof} First, note that Condition (2a), Condition (2b), and Condition (2c) of Lemma~\ref{setup} imply that $\beta$ cannot be positive. Because positive syllables of $\beta$ are cyclically neighbored by negative syllables, we may decompose $\beta$ as $\beta=N_{1}$ (in the case that $\beta$ is a negative braid) or as $\beta=P_{1}N_{1} \cdots P_{t}N_{t}$, where $P_{i}$ denotes a positive syllable of $\beta$ and $N_{i}$ denotes a maximal length negative induced subword of $\beta$. 

\bigskip

\noindent \underline{Small Inner Circles}: Let $\sigma_{m_{i}}^{r_{i}}$, where $r_{i} \leq -3$, denote a negative syllable. Then, except for the additional $n-2$ surrounding vertical line segments, the A-resolution of this syllable will look like the left side of Figure~\ref{longshort}. In particular, having $r_{i} \leq -3$ implies the existence of at least two small inner circles. 

\bigskip

\noindent \underline{Wandering Circles}: Let $\sigma_{m_{i}}^{r_{i}}\sigma_{m_{i+1}}^{r_{i+1}}=\sigma_{{m_i}}^{r_{i}+1}(\sigma_{m_{i}}^{-1}\sigma_{m_{i+1}}^{-1})\sigma_{m_{i+1}}^{r_{i+1}+1}$, where $r_{i} \leq -3$ and $r_{i+1} \leq -3$, denote a pair of adjacent negative syllables.

\bigskip

\underline{Case 1}: Suppose $m_{i+1}=m_{i}-1$ or $m_{i+1}=m_{i}+1$. Then the pair of adjacent negative syllables involve adjacent generators of $B_{n}$. Consequently, the corresponding portion of $H_{A}$ will resemble (up to reflection and except for the additional surrounding vertical line segments) the right side of Figure~\ref{loop}. The key feature of this figure is the fact that we see a portion of a wandering circle, where the wandering behavior corresponds to the existence of the $\sigma_{m_{i}}^{-1}\sigma_{m_{i+1}}^{-1}$ induced subword. 

\bigskip

\underline{Case 2}: Suppose $\left|m_{i+1}-m_{i}\right| \geq 2$. In this case, the pair of adjacent negative syllables involve far commuting generators of $B_{n}$. Then, except for the additional $n-4$ surrounding vertical line segments, we get that the A-resolution will look like two copies of the left side of Figure~\ref{longshort}. Thus, we return to the case of small inner circles. 

\bigskip

\noindent \underline{Medium Inner Circles and Nonwandering Circles}: Let $\sigma_{m_{i}}^{r_{i}}$, where $r_{i}>0$, denote a positive syllable of $\beta$ that is cyclically neighbored on both sides by negative syllables as in Remark~\ref{cond}. Then the A-resolution of the induced subword will resemble either the left side (up to reflection and except for the additional surrounding vertical line segments) or the center (except for the additional surrounding vertical line segments) of Figure~\ref{loop}. In particular, the existence of an isolated positive syllable corresponds to the existence of one or two medium inner circles, one in the case of Condition~(2a) and Condition~(2c) and two in the case of Condition~(2b). Furthermore, a nonwandering circle will occur in $H_{A}$ precisely when all syllables in the generator $\sigma_{1}$ are positive (see the left side of Figure~\ref{loop}) or when all syllables in the generator $\sigma_{n-1}$ are positive (which would give a reflection of the left side of Figure~\ref{loop}).
 
\bigskip

Note that we may classify the wandering circles in the annulus into essential and non-essential circles. Since the A-resolutions of all portions of the closure of $\beta=N_{1}$ and $\beta=P_{1}N_{1} \cdots P_{t}N_{t}$ have been considered locally and since gluing such portions together joins wandering and potential nonwandering circle portions together to form wandering and nonwandering circles, then we have the desired classification of all-A circles.  
\end{proof}

\begin{sublemma} \label{sublemma2} Let $D(K)=\widehat{\beta}$ satisfy the assumptions of Lemma~\ref{setup}. Then $D(K)$ is connected, prime, and contains $t(D) \geq 2(n-1)$ twist regions.    
\end{sublemma}

\begin{proof} Recall that $D(K)$ is connected if and only if the projection graph of $D(K)$ is path-connected. Since $\beta \in B_{n}$ contains a complete subword, then each generator of $B_{n}$ (each of which corresponds to a crossing of $D(K)$ between adjacent braid string portions) must occur at least once. This fact implies that the closed $n$-braid diagram $D(K)$ must be connected. 

Since $\beta$ contains two disjoint complete subwords, then it must be that each of the generators $\sigma_{1}, \ldots, \sigma_{n-1}$ of $B_{n}$ must occur at least twice in (two distinct syllables of) the braid word. Since the syllables of the cyclically reduced braid word $\beta$ correspond to the twist regions of $D(K)$, then we have that $t(D) \geq 2(n-1)$. 

Let $C$ be a simple closed curve that intersects $D(K)=\widehat{\beta}$ exactly twice (away from the crossings) and contains crossings on both sides. To show that $D(K)$ is prime, we need to show that such a curve $C$ cannot exist. Recall that $D(K)$ lies in an annular region of the plane. 

%See Figure~\ref{inducedschematic} for a schematic depiction of this situation. 
%
%\begin{figure}
	%\centering
	  %\def\svgwidth{1.5in}
		  %\input{inducedschematic.pdf_tex}
	%\caption{A schematic depiction of a braid $\beta \in B_{n}$ that contains two disjoint induced complete subwords $\gamma_{1}$ and $\gamma_{2}$.}
	%\label{inducedschematic}
%\end{figure}

\bigskip

\noindent \underline{Case 1}: Suppose $C$ contains a point $p$ outside of this annular region. Since $C$ only intersects $D(K)$ twice and must start and end at $p$, then $C$ must intersect $D(K)$ twice in braid string position $1$ or twice in braid string position $n$. This cannot happen because the crossings corresponding to the occurrences of $\sigma_{1}$ and $\sigma_{n-1}$ prevent $C$ from being able to close up in a way that will contain crossings on both sides. 

\bigskip

\noindent \underline{Case 2}: Suppose $C$ contains a point $p$ between braid string positions $i$ and $i+1$ for some $1 \leq i \leq n-1$. See Figure~\ref{primeclose}. Since $C$ only intersects $D(K)$ twice and must start and end at $p$, then $C$ must intersect $D(K)$ twice in braid string position $i$ or twice in braid string position $i+1$. Suppose, as in Figure~\ref{primeclose}, that $C$ intersects $D(K)$ twice in braid string position $i$. The case for braid string position $i+1$ is very similar. If $i=1$, then $C$ contains a point outside of the annular region and we return to Case 1. Suppose $2 \leq i  \leq n-1$. Since $\beta$ contains two disjoint induced complete subwords, then the generator $\sigma_{i-1}$ and the generator $\sigma_{i}$ must each occur at least twice, once above $p$ and once below $p$. These occurrences prevent $C$ from being able to close up in a way that will contain crossings on both sides. Thus, this case cannot occur. 
\end{proof}

\begin{figure}
	\centering
	  \def\svgwidth{1.5in}
		  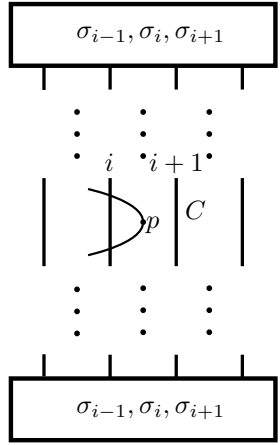
	\caption{A simple closed curve $C$ intersecting $D(K)$ twice and containing a point $p$ between braid string positions $i$ and $i+1$. Each box in the figure represents the eventual occurrence of the generator $\sigma_{i}$, the generator $\sigma_{i-1}$ (if it exists), and the generator $\sigma_{i+1}$ (if it exists).}
	\label{primeclose}
\end{figure}

\begin{sublemma} \label{sublemma3} Let $D(K)=\widehat{\beta}$ satisfy the assumptions of Lemma~\ref{setup}. Then $D(K)$ is A-adequate.  
\end{sublemma}

\begin{proof} To prove that $D(K)$ is A-adequate, we need to show that no A-segment of the all-A state $H_{A}$ joins an all-A circle to itself. Note that positive syllables in $\beta$ (which include the short twist regions of $D(K)$) A-resolve to give horizontal A-segments and that negative syllables in $\beta$ (which include the long twist regions of $D(K)$) A-resolve to give vertical A-segments. Suppose, for a contradiction, that an A-segment joins an all-A circle to itself. 

\bigskip

\noindent \underline{Case 1}: Suppose the A-segment is a vertical segment. Condition (1) (that negative exponents are at least three in absolute value) implies that it is impossible for a vertical A-segment to join an all-A circle to itself. This is because all vertical A-segments either join distinct small inner circles or join a small inner circle to a medium inner circle or a wandering circle. (See the left side of Figure~\ref{longshort}.)

\bigskip  

\noindent \underline{Case 2}: Suppose the A-segment is a horizontal segment. By Remark~\ref{cond} and Figure~\ref{loop}, it can be seen that it is impossible for a horizontal A-segment to join an all-A circle to itself. This is because all horizontal A-segments join a medium inner circle either to another medium inner circle, to a wandering circle, or to a nonwandering circle. 
\end{proof}

\begin{sublemma} \label{sublemma4} Let $D(K)=\widehat{\beta}$ satisfy the assumptions of Lemma~\ref{setup}. Then $D(K)$ satisfies the TELC.
\end{sublemma}

\begin{proof} Let $C_{1}$ and $C_{2}$ be two distinct all-A circles that share a pair of distinct A-segments, call them $s_{1}$ and $s_{2}$. To prove that $D(K)$ satisfies the TELC, we need to show that these A-segments correspond to crossings from the same short twist region of $D(K)$.

\bigskip

\noindent \underline{Case 1}: Suppose one of $s_{1}$ and $s_{2}$ is a horizontal segment (that must come from a cyclically isolated positive syllable of $\beta$). Then it is impossible for the other segment to be a vertical segment. This is because one of $C_{1}$ and $C_{2}$, say $C_{1}$, must be a medium inner circle. This circle is adjacent, via horizontal A-segments, to the second circle $C_{2}$, which will either be another medium inner circle, a wandering circle, or a nonwandering circle. By Condition~(1), the circle $C_{1}$ is also adjacent to small inner circles above and below. Consequently, the second segment cannot be vertical. Thus, if one of the segments $s_{1}$ and $s_{2}$ is horizontal, then the other segment must also be horizontal. Since the horizontal segments incident to the medium inner circle $C_{1}$ necessarily belong to the same resolution of a short twist region of $D(K)$, then the TELC is satisfied. (See the left side and center of of Figure~\ref{loop}.) 

\bigskip

\noindent \underline{Case 2}: Suppose both $s_{1}$ and $s_{2}$ are vertical segments. Since all vertical A-segments either join distinct small inner circles or join a small inner circle to a medium inner circle or a wandering circle, then one of $C_{1}$ and $C_{2}$ must be a small inner circle. By Condition~(1), it is impossible for a small inner circle to share more than one A-segment with another all-A circle. (See the left side of Figure~\ref{longshort}.) Therefore, this case cannot occur.
\end{proof}

We now combine all of the necessary ingredients in order to prove the Main Lemma. 

\begin{proof}[Proof of Lemma~\ref{setup} (The Main Lemma)] By applying Sublemma~\ref{sublemma1}, Sublemma~\ref{sublemma2}, Sublemma~\ref{sublemma3}, and Sublemma~\ref{sublemma4}, we get every conclusion of Lemma~\ref{setup} except the hyperbolicity of $K$, which follows from Theorem~\ref{Cor}. 
\end{proof}

\subsection{Diagrammatic volume bounds}

In this section, we apply Lemma~\ref{setup} and Theorem~\ref{Cor} to produce diagrammatic volume bounds for the complements of the closed $n$-braids considered in this paper. We now begin with a lemma for the case of 3-braids. 

%To find lower bounds on volume for the closed braids considered in this paper, we need the following lemma from \cite{Me}: 
%
%\begin{lemma}[\cite{Me}] \label{chilemma} Let $D(K)$ be a connected A-adequate link diagram that satisfies the TELC. Then we have that
%\begin{equation*}
%-\chi(\mathbb{G}_{A}')=t(D)-\#\left\{\mathrm{OCs}\right\}.
%\end{equation*}
%\end{lemma}

\begin{lemma} Let $D(K)=\widehat{\beta}$ denote the closure of a nice 3-braid $\beta \in B_{3}$. Furthermore, assume that the positive syllables of $\beta$ are cyclically neighbored on both sides by negative syllables. Then the all-A state $H_{A}$ of $D(K)$ satisfies precisely one of the following conditions:
\label{mylemma2}
\begin{enumerate}
\item[(1)] $H_{A}$ contains exactly one nonwandering circle and no wandering circles.
\item[(2)] $H_{A}$ contains exactly one wandering circle and no nonwandering circles. 
\end{enumerate}
\end{lemma}

\begin{proof} A 3-braid $\beta \in B_{3}$ is either alternating or nonalternating. 

\bigskip

\noindent \underline{Case 1}: Suppose $\beta$ is alternating. Then one of the braid generators must always occur with positive exponents and the other generator must always occur with negative exponents. Hence, by the proof of Sublemma~\ref{sublemma1}, since a generator occurs with only positive exponents, then there will be a nonwandering circle. Since only one generator occurs with only positive exponents, then there is only one such nonwandering circle. Also, since adjacent negative syllables cannot occur in $\beta$, then wandering circles cannot occur in $H_{A}$. 

\bigskip

\noindent \underline{Case 2}: Suppose $\beta$ is nonalternating. Since positive syllables are cyclically neighbored on both sides by negative syllables, then the nonalternating behavior of $\beta$ must come from a pair of adjacent negative syllables. Hence, by the proof of Sublemma~\ref{sublemma1}, this implies the existence of a wandering circle in $H_{A}$ and prevents the existence of a nonwandering circle (since both generators occur once with negative exponent). Finally, since a wandering circle (which is always essential in the case that $n=3$) ``uses up'' a braid string from the braid closure and since a wandering circle must wander from braid string position 1 to braid string position 3 and back before closing up, then the existence of a second such wandering circle is impossible. 

\end{proof}

\begin{definition}
Let $D(K)=\widehat{\beta}$ be the closure of an $n$-braid $\beta \in B_{n}$. With the sign conventions for braid generators given in Figure~\ref{braidgen}, we call a twist region \emph{positive} if its crossings correspond to a positive syllable in the braid word $\beta$ and \emph{negative} if its crossings correspond to a negative syllable in the braid word $\beta$. Let $t^{+}(D)$ denote the number of positive twist regions in $D(K)$ and let $t^{-}(D)$ denote the number of negative twist regions in $D(K)$.
\end{definition}

The following theorem, which is the first main result of this paper, is a more precise version of Theorem~\ref{introthm} from the introduction. 

\begin{theorem} \label{mytheorem} Let $D(K)=\widehat{\beta}$, where $\beta \in B_{n}$ satisfies the assumptions of Lemma~\ref{setup}. Let $m$ denote the number of non-essential wandering circles in the all-A state $H_{A}$ of $D(K)$. In the case that $n=3$, we have that
$$-\chi(\mathbb{G}_{A}')=t^{-}(D)-1 \geq \frac{1}{2}\cdot (t(D)-2),$$
\noindent which gives the volume bounds
$$\frac{v_{8}}{2}\cdot(t(D)-2) \leq v_{8}\cdot(t^{-}(D)-1) \leq \mathrm{vol}(S^{3}\backslash K) < 10v_{3}\cdot(t(D)-1).$$
In the case that $n\geq4$ and that the only positive syllables of $\beta$ possibly occur in the generators $\sigma_{1}$ and $\sigma_{n-1}$, we have the volume bounds
$$\frac{v_{8}}{2}\cdot(t(D)-2(n+m-2)) \leq v_{8}\cdot(t^{-}(D)-(n+m-2)) \leq \mathrm{vol}(S^{3}\backslash K) < 10v_{3}\cdot(t(D)-1).$$
In the case that $n\geq4$ and that positive syllables of $\beta$ occur in a generator $\sigma_{i}$ for some $2 \leq i \leq n-2$, we have the volume bounds
$$v_{8}\cdot(t^{-}(D)-t^{+}(D)-(n+m-2)) \leq \mathrm{vol}(S^{3}\backslash K) < 10v_{3}\cdot(t(D)-1) = 10v_{3}\cdot(t^{-}(D)+t^{+}(D)-1).$$
\end{theorem}

It is worth noting that the lower bounds on volume in terms of $t(D)$ given above are sharper than the general lower bounds given in the Main Theorem (Theorem 1.1) of \cite{Me} when $t(D)$ is large compared to the sum $n+m$. Specifically, the lower bounds are sharper when $t(D)>4$ for the case when $n=3$ (in which $m=0$ and $t(D) \geq 4=2(n-1)$) and when $t(D)>6(n+m)-16$ for the case when $n \geq 4$.  

\begin{proof}[Proof of Theorem~\ref{mytheorem}] To begin, note that we may apply the many results of Lemma~\ref{setup}. Since $D(K)$ is connected, A-adequate, and satisfies the TELC, then Lemma~3.4 of \cite{Me} implies that $$-\chi(\mathbb{G}_{A}')=t(D)-\#\left\{\text{OCs}\right\},$$ where $\#\left\{\text{OCs}\right\}$ is the number of all-A circles in $H_{A}$, called \emph{other circles}, that are not small inner circles. Note that, by Sublemma~\ref{sublemma1}, the set of other circles consists of the medium inner circles, essential wandering circles, non-essential wandering circles, and nonwandering circles in $H_{A}$. By Remark~\ref{cond} and Figure~\ref{loop}, it can be seen that $$t^{+}(D) \leq \#\left\{\text{medium\ inner\ circles}\right\} \leq 2t^{+}(D),$$ which says that each positive twist region corresponds to either one or two medium inner circles. In particular, looking at the left side (up to reflection) of Figure~\ref{loop}, it can be seen that $$\#\left\{\text{medium\ inner\ circles}\right\} = t^{+}(D)$$ in the case that $n=3$ and in the case that $n \geq 4$ when the only positive syllables of $\beta$ possibly occur in the generators $\sigma_{1}$ and $\sigma_{n-1}$. Additionally, Condition (2a), Condition (2b), and Condition (2c) of Lemma~\ref{setup} imply that at least half of the twist regions in $D(K)$ must be negative twist regions. 

\bigskip

\noindent \underline{Case 1}: Suppose $n=3$. By Lemma~\ref{mylemma2}, we know that the total number of wandering circles and nonwandering circles is one. Therefore, using all of what was said above, we get
\begin{eqnarray*}-\chi(\mathbb{G}_{A}') & = & t(D)-\#\left\{\text{OCs}\right\}\\
\ & = & t(D)-\#\left\{\text{medium\ inner\ circles}\right\}-\#\left\{\text{wandering\ circles}\right\}\\
\ & \ & \hspace{2in} -\#\left\{\text{nonwandering\ circles}\right\}\\
\ & = & t(D)-t^{+}(D)-1\\
\ & = & t^{-}(D)-1\\
\ & \geq & \frac{t(D)}{2}-1\\
\ & = & \frac{1}{2}\cdot (t(D)-2).
\end{eqnarray*} 

\noindent \underline{Case 2}:  Suppose $n\geq4$. Note that both essential wandering circles and nonwandering circles ``use up'' a braid string from the braid closure. Also, note that Lemma~\ref{setup} implies that there can be at most two nonwandering circles in the all-A state $H_{A}$. These two facts imply that

\begin{equation}\label{EWC}
\#\left\{\text{essential\ wandering\ circles}\right\}+\#\left\{\text{nonwandering\ circles}\right\} \leq n-2.
\end{equation}

\bigskip

\noindent \underline{Subcase 1}: Suppose that the only positive syllables of $\beta$ possibly occur in the generators $\sigma_{1}$ and $\sigma_{n-1}$, in which case we have that $\#\left\{\text{medium\ inner\ circles}\right\} = t^{+}(D)$. Then, by using Inequality~\ref{EWC} and what was said in the beginning of the proof, we get
\begin{eqnarray*}-\chi(\mathbb{G}_{A}') & = & t(D)-\#\left\{\text{OCs}\right\}\\
\ & = & t(D)-\#\left\{\text{medium\ inner\ circles}\right\}-\#\left\{\text{non-essential\ wandering\ circles}\right\}\\
\ & \ & \hspace{.318in} -\#\left\{\text{essential\ wandering\ circles}\right\}-\#\left\{\text{nonwandering\ circles}\right\}\\
\ & \geq & t(D)-t^{+}(D)-m-(n-2)\\
\ & = & t^{-}(D)-(n+m-2)\\
\ & \geq & \frac{t(D)}{2}-(n+m-2)\\
\ & = & \frac{1}{2}\cdot (t(D)-2(n+m-2)).
\end{eqnarray*}

\noindent \underline{Subcase 2}: Suppose that positive syllables of $\beta$ occur in a generator $\sigma_{i}$ for some $2 \leq i \leq n-2$, in which case we have that $\#\left\{\text{medium\ inner\ circles}\right\} \leq 2t^{+}(D)$. Then, by using Inequality~\ref{EWC} and what was said in the beginning of the proof, we get
\begin{eqnarray*}-\chi(\mathbb{G}_{A}') & = & t(D)-\#\left\{\text{OCs}\right\}\\
\ & = & t(D)-\#\left\{\text{medium\ inner\ circles}\right\}-\#\left\{\text{non-essential\ wandering\ circles}\right\}\\
\ & \ & \hspace{.318in} -\#\left\{\text{essential\ wandering\ circles}\right\}-\#\left\{\text{nonwandering\ circles}\right\}\\
\ & \geq & t(D)-2t^{+}(D)-m-(n-2)\\
\ & = & t^{-}(D)-t^{+}(D)-(n+m-2).\\
\end{eqnarray*}

\noindent By applying the inequalities for $-\chi(\mathbb{G}_{A}')$ above to Theorem~\ref{Cor}, we get the desired volume bounds.

\end{proof}

\subsection{Volume bounds in terms of the colored Jones polynomial}

We now continue our study of volume bounds for hyperbolic A-adequate closed braids by translating our diagrammatic volume bounds from earlier to volume bounds in terms of the stable penultimate coefficient $\beta_{K}'$ of the colored Jones polynomial. Recall from the introduction that we denote the \emph{$j^{th}$ colored Jones polynomial} of a link $K$ by
$$J_{K}^{j}(t)=\alpha_{j}t^{m_{j}}+\beta_{j}t^{m_{j}-1}+\cdots+\beta_{j}'t^{r_{j}+1}+\alpha_{j}'t^{r_{j}},$$
\noindent where $j \in \mathbb{N}$ and where the degree of each monomial summand decreases from left to right. The following corollary of Theorem~\ref{mytheorem} is a more precise version of Corollary~\ref{newcor} from the introduction.
 
\begin{corollary} 
\label{newestsetup}
Let $D(K)=\widehat{\beta}$ denote the closure of a nice $n$-braid $\beta \in B_{n}$ that satisfies the assumptions of Lemma~\ref{setup}. In the case that $n=3$ and in the case that $n \geq 4$ when the only positive syllables of $\beta$ possibly occur in the generators $\sigma_{1}$ and $\sigma_{n-1}$, we have the volume bounds 
$$v_{8}\cdot(\left|\beta_{K}'\right|-1) \leq \mathrm{vol}(S^{3}\backslash K) < 20v_{3}\cdot\left(\left|\beta_{K}'\right|+n+m-\dfrac{7}{2}\right).$$
\end{corollary}

\begin{proof} Since $D(K)$ is a connected, A-adequate link diagram, then Theorem 3.1 of \cite{HeadTail} implies that the absolute value $$\left|\beta_{K}'\right|:=\left|\beta_{j}'\right|=1-\chi(\mathbb{G}_{A}')$$ is independent of $j \geq 2$. By combining this result with Theorem~\ref{Cor} and Theorem~\ref{mytheorem}, we get the desired result. 
\end{proof}

%\begin{proof} The first conclusions of the theorem follow from Lemma~\ref{setup}. Combining Theorem~\ref{stable} with Theorem~\ref{Cor}, we get that (for all $n \geq 3$): 
%
%$$v_{8}\cdot(\left|\beta_{K}'\right|-1) = -v_{8} \cdot \chi(\mathbb{G}_{A}') \leq \mathrm{vol}(S^{3}\backslash K).$$ 
%
%\noindent Consider the case when $n=3$. Combining Theorem~\ref{stable} with Theorem~\ref{mytheorem} gives: 
%
%$$\displaystyle \left|\beta_{K}'\right|=1-\chi(\mathbb{G}_{A}') \geq 1+\frac{1}{2}\cdot (t(D)-1)-\frac{1}{2}=\frac{t(D)}{2},$$ 
%
%\noindent which implies that:
%
%$$t(D)-1 \leq 2\cdot(\left|\beta_{K}'\right|-1)+1.$$
%
%\noindent Applying this inequality to Theorem~\ref{Cor}, we get:
%
%$$\mathrm{vol}(S^{3}\backslash K) < 10v_{3}\cdot(t(D)-1) \leq 20v_{3}\cdot(\left|\beta_{K}'\right|-1)+10v_{3}.$$ 
%
%\noindent This gives the first desired set of volume bounds. 
%
%\bigskip
%
%\noindent Now consider the case when $n \geq 4$. Combining Theorem~\ref{stable} with Theorem~\ref{mytheorem} gives:
%\begin{eqnarray*}
%\left|\beta_{K}'\right|=1-\chi(\mathbb{G}_{A}') & \geq & 1+\frac{1}{2}\cdot (t(D)-1)-\frac{1}{2}\cdot (2(n+m)-5)\\
%\ &  =  & \frac{t(D)-2(n+m)+6}{2},
%\end{eqnarray*}
%
%\noindent which implies that:
%
%$$t(D)-1 \leq 2\cdot(\left|\beta_{K}'\right|-1)+2(n+m)-5.$$
%
%\noindent Applying this inequality to Theorem~\ref{Cor}, we get:
%
%$$\mathrm{vol}(S^{3}\backslash K) < 10v_{3}\cdot(t(D)-1) \leq 20v_{3}\cdot(\left|\beta_{K}'\right|-1)+10v_{3}\cdot (2(n+m)-5).$$
 %
%\noindent This gives the second desired set of volume bounds. \end{proof}

\section{Volume bounds for A-adequate closed 3-braids in terms of the Schreier normal form}

% INTRODUCE SECTION!!!

\subsection{The Schreier normal form for 3-braids and volume}
\label{normformsec}

% INTRODUCE SECTION!!!

A useful development in the history of 3-braids was the solution to the Conjugacy Problem (\cite{Schreier}). During this time, an algorithm was developed that produces from an arbitrary 3-braid word $\beta$ a conjugate 3-braid word $\beta'$, called the \emph{Schreier normal form}, which is the unique representative of the conjugacy class of $\beta$. A version of this algorithm is presented below. 

%Given such an algorithm, Birman and Menasco (\cite{BirmanMenasco}) gave a complete classification of links that can be represented as 3-braid closures. To be more specific they showed that, up to an explicit list of exceptions, each 3-braid closure comes from a single conjugacy class. Hence, up to some exceptions, the normal form of the braid $\beta$ determines the unique link type of the closed braid $\widehat{\beta}$. 
%
%Because it will be needed in what follows, we now present the algorithm (adapted from Section 7.1 of \cite{BirmanMenasco}) that, given a 3-braid $\beta \in B_{3}$, produces the Schreier normal form $\beta'$ of $\beta$. 
%
%\begin{remark}
%It is important to note that cyclic permutation (a special case of conjugacy) may be needed during the steps of this algorithm.  
%\end{remark}

\bigskip

\noindent \textbf{Schreier Normal Form Algorithm:} 

\begin{itemize}
\item[(1)] Let $\beta \in B_{3}=\langle \sigma_{1}, \sigma_{2}\ \vert\ \sigma_{1}\sigma_{2}\sigma_{1}=\sigma_{2}\sigma_{1}\sigma_{2}\rangle$ be cyclically reduced into syllables. Introduce new variables $x=\left(\sigma_{1}\sigma_{2}\sigma_{1}\right)^{-1}$ and $y=\sigma_{1}\sigma_{2}$. Thus we have that $\sigma_{1}=y^2x$, $\sigma_{2}=xy^2$, $\sigma_{1}^{-1}=xy$, and $\sigma_{2}^{-1}=yx$. Possibly using cyclic permutation, rewrite $\beta$ as a cyclically reduced word (that is positive) in $x$ and $y$. 
\item[(2)] Introduce $C=x^{-2}=y^{3} \in Z(B_{3}),$ where $Z(B_{3})$ denotes the center of the 3-braid group $B_{3}$. By using these relations and the commutativity of $C$ as much as possible, group all powers of $C$ at the beginning of the braid word and reduce the exponents of $x$ and $y$ as much as possible. Rewrite $\beta$ as $\beta=C^{j}\eta$, where $j \in \mathbb{Z}$ and where
$$\eta = \left\{
\begin{array}{ll} 
\left(xy\right)^{p_{1}}\left(xy^2\right)^{q_{1}}\cdots \left(xy\right)^{p_{s}}\left(xy^2\right)^{q_{s}} & \mathrm{for\ some}\ s, p_{i}, q_{i} \geq 1 \\
\left(xy\right)^{p} & \mathrm{for\ some}\ p \geq 1 \\
\left(xy^{2}\right)^{q} & \mathrm{for\ some}\ q \geq 1 \\
y & \ \\
y^2 & \ \\
x & \ \\
1 & \ \\
\end{array}  \right.
$$
\item[(3)] Possibly using cyclic permutation and the commutativity of $C$, rewrite $\beta$ back in terms of $\sigma_{1}$ and $\sigma_{2}$ as $\beta'=C^{k}\eta'$, where $k \in \mathbb{Z}$ and where
$$ \eta' = \left\{
\begin{array}{ll} 
\sigma_{1}^{-p_{1}}\sigma_{2}^{q_{1}}\cdots \sigma_{1}^{-p_{s}}\sigma_{2}^{q_{s}} & \mathrm{for\ some}\ s, p_{i}, q_{i} \geq 1 \\
\sigma_{1}^{p} & \mathrm{for\ some}\ p\in \mathbb{Z} \\
\sigma_{1}\sigma_{2} & \ \\
\sigma_{1}\sigma_{2}\sigma_{1} & \ \\
\sigma_{1}\sigma_{2}\sigma_{1}\sigma_{2} & \ 
\end{array}  \right.
$$
\end{itemize}

\begin{definition} We call $\beta'=C^{k}\eta' \in B_{3}$ the \emph{Schreier normal form} of $\beta \in B_{3}$. The braid word $\beta'$ is the unique representative of the conjugacy class of $\beta$. Following \cite{Cusp}, we will call a braid $\beta$ \emph{generic} if it has Schreier normal form 
$$\beta'=C^{k}\sigma_{1}^{-p_{1}}\sigma_{2}^{q_{1}}\cdots \sigma_{1}^{-p_{s}}\sigma_{2}^{q_{s}}.$$ 
\end{definition}

Using the Schreier normal form of a 3-braid, Futer, Kalfagianni, and Purcell (\cite{Cusp}) classified the hyperbolic 3-braid closures. Furthermore, given such a hyperbolic closed 3-braid, they gave two-sided bounds on the volume of the link complement, expressing the volume in terms of the parameter $s$ from the Schreier normal form of the 3-braid. By using the more recent machinery built by the same authors in \cite{Guts}, we aim to obtain a sharper lower bound on volume. To begin our study, we first recall two propositions from \cite{Cusp}. 

\begin{proposition}[\cite{Cusp}, Theorem 5.5] Let $D(K)=\widehat{\beta}$ denote the closure of a 3-braid $\beta \in B_{3}$. Then $K$ is hyperbolic if and only if 
\begin{itemize}
\item[(1)] $\beta$ is generic, and 
\item[(2)] $\beta$ is not conjugate to $\sigma_{1}^p\sigma_{2}^q$ for any integers $p$ and $q$. 
\end{itemize}
\label{hypgeneric}
\end{proposition}

\begin{proposition}[\cite{Cusp}, Theorem 5.6]\label{svolbound}
Let $D(K)=\widehat{\beta}$ denote the closure of a 3-braid $\beta \in B_{3}$. Then, assuming that $K$ is hyperbolic, we have that
$$4v_{3}\cdot s-276.6 < \mathrm{vol}(S^3 \backslash K) < 4v_{8}\cdot s.$$ 
\end{proposition}

%Given the generic normal form above, we shall view all of the subwords $\sigma_{1}^{-p_{i}}$, $\sigma_{2}^{q_{i}}$, and $C^{k}$ as \emph{generalized twist regions}. Note that, in this definition, the subwords $\sigma_{1}^{-p_{i}}$ and $\sigma_{2}^{q_{i}}$ correspond to the standard definition of a twist region and the subword $C^{k}$ can be viewed as a three-stranded twist region. Let $t_{\mathrm{gen}}(D)$ denoted the number of generalized twist regions of a a closed 3-braid diagram $D=\widehat{\beta}$. Then we get the following result:
%
%\begin{corollary}[\cite{Cusp}, Corollary 5.7]
%Given the same assumptions as in the previous theorem, we have:
%
%$$2v_{3}t_{\mathrm{gen}}(D)-279 < \mathrm{vol}(S^3-K) < 2v_{8}t_{\mathrm{gen}}(D).$$
%
%\end{corollary}

\subsection{Volume bounds in terms of the Schreier normal form}

We begin this section by presenting a result that, for the family of closed 3-braids that satisfy Lemma~\ref{setup}, states that the parameters $k$ and $s$ from the Schreier normal form of the 3-braid can actually be read off of the original 3-braid word. We then use this result to obtain volume bounds in terms of the parameter $s$ from the Schreier normal form. 

\begin{theorem}
\label{normformthm}
Let $D(K)=\widehat{\beta}$ denote the closure of a nice $3$-braid $\beta \in B_{3}$ that satisfies the assumptions of Lemma~\ref{setup}. Then $\beta$ is generic and, furthermore, we are able to express the parameters $k$ and $s$ of the Schreier normal form $\beta'$ in terms of the original 3-braid $\beta$ as follows:
\begin{enumerate}
\item[(1)] $k=-\# \left\{\text{induced\ products}\ \sigma_{2}^{n_{2}}\sigma_{1}^{n_{1}}\ \text{of\ negative\ syllables\ of}\ \beta,\ \text{where}\ n_{1}, n_{2} \leq -3 \right\}$.
\item[(2)] $s=t^{-}(D)=\#\left\{\text{negative\ syllables\ in}\ \beta \right\}$.
\end{enumerate}
\end{theorem}

Note that, when looking for the induced products $\sigma_{2}^{n_{2}}\sigma_{1}^{n_{1}}$ of negative syllables of $\beta$ to find $k$, we must look cyclically in the braid word. As a special case of the above theorem, notice that if $\beta=\sigma_{i}^{p_{1}}\sigma_{j}^{n_{1}}\cdots\sigma_{i}^{p_{l}}\sigma_{j}^{n_{l}}$ is an alternating 3-braid where $\{i,j\}=\{1,2\}$, then $k=0$ and $s=l$. 

%\bigskip
%
%\noindent \textbf{Example:} Consider the 3-braid $\beta=\sigma_{1}^{3}\sigma_{2}^{-3}\sigma_{1}^{-5}\sigma_{2}^{-3}$. Applying the Schreier Normal Form Algorithm, we get
%\begin{eqnarray*}
%\beta & = &\sigma_{1}^{3}\sigma_{2}^{-3}\sigma_{1}^{-5}\sigma_{2}^{-3}\\
%\ & = & (y^{2}x)^{3}(yx)^{3}(xy)^{5}(yx)^{3}\\
%\ & = & y^{2}(xy^{2})^{2}(xy)^{3}(x^{2})y(xy)^{3}(xy^{2})(xy)^{2}x\\
%\ & = & y^{2}(xy^{2})^{2}(xy)^{3}(C^{-1})y(xy)^{3}(xy^{2})(xy)^{2}x\\
%\ & \cong & C^{-1}y^{2}(xy^{2})^{2}(xy)^{2}(xy^{2})(xy)^{3}(xy^{2})(xy)^{2}x\\
%\ & \cong & C^{-1}(xy^{2})^{3}(xy)^{2}(xy^{2})(xy)^{3}(xy^{2})(xy)^{2}\\
%\ & \cong & C^{-1}(xy)^{2}(xy^{2})^{3}(xy)^{2}(xy^{2})(xy)^{3}(xy^{2})\\
%\ & = & C^{-1}\sigma_{1}^{-2}\sigma_{2}^{3}\sigma_{1}^{-2}\sigma_{2}\sigma_{1}^{-3}\sigma_{2}\\
%\ & = & \beta',
%\end{eqnarray*}
%\noindent where $\cong$ denotes that either cyclic permutation or the fact that $C \in Z(B_{3})$ has been used. Thus, we see that
%$$k=-1=-\# \left\{\text{induced\ products}\ \sigma_{2}^{n_{2}}\sigma_{1}^{n_{1}}\ \text{of\ negative\ syllables\ of}\ \beta,\ \text{where}\ n_{1}, n_{2} \leq -3 \right\}$$
%\noindent and
%$$s=3=t^{-}(D)=\#\left\{\text{negative\ syllables\ in}\ \beta \right\}.$$

Also note that the count $$\# \left\{\text{induced\ products}\ \sigma_{2}^{n_{2}}\sigma_{1}^{n_{1}}\ \text{of\ negative\ syllables\ of}\ \beta,\ \text{where}\ n_{1}, n_{2} \leq -3 \right\}$$ is the same as the count $$\# \left\{\text{induced\ products}\ \sigma_{1}^{n_{1}'}\sigma_{2}^{n_{2}'}\ \text{of\ negative\ syllables\ of}\ \beta,\ \text{where}\ n_{1}', n_{2}' \leq -3 \right\}.$$ This is because, recalling from Lemma~\ref{setup} that wandering corresponds to the existence of adjacent negative syllables (in adjacent braid generators), the number of times a wandering circle wanders from braid string position 1 to braid string position 3 is the same as the number of times a wandering circle wanders from braid string position 3 to braid string position 1.  

%Recall that, because positive syllables are cyclically neighbored by negative syllables in a braid $\beta$ that satisfies the assumptions of Lemma~\ref{setup}, then we may decompose $\beta$ as $\beta=N_{1}$ (in the case that $\beta$ is a negative braid) or as $\beta=P_{1}N_{1} \cdots P_{t}N_{t}$, where $P_{i}$ denotes a positive syllable of $\beta$ and $N_{i}$ denotes a maximal length negative induced subword of $\beta$. 

\begin{remark}
Since the 3-braids of Theorem~\ref{normformthm} (which satisfy the assumptions of Lemma~\ref{setup}) are generic, then Proposition~4.15 of \cite{Width} implies, with the assumption that $k \neq 0$, that $$\left|k\right|-1 \leq g_{T}(K) \leq \left|k\right|,$$ where $K$ is the hyperbolic link with closed 3-braid diagram $D(K)=\widehat{\beta}$ and where $g_{T}(K)$ denotes the Turaev genus of $K$. Thus, since conclusion (1) of Theorem~\ref{normformthm} gives that we can read the parameter $k$ from the original 3-braid $\beta$, then we can visually bound the Turaev genus of the braid closure.   
\end{remark}

Deferring the proof of Theorem~\ref{normformthm} to the next section, we now relate the parameter $s$ from the Schreier normal form of a 3-braid from Lemma~\ref{setup} to the stable penultimate coefficient $\beta_{K}'$ of the colored Jones polynomial of the closure of this 3-braid. 

\begin{corollary}\label{scor}
Let $D(K)=\widehat{\beta}$, where $\beta \in B_{3}$ satisfies the assumptions of Lemma~\ref{setup}. Then
$$s=t^{-}(D)=\vert\beta_{K}'\vert.$$ 
Thus, $t^{-}(D)$ and $s$ are link invariants. 
\end{corollary} 

\begin{proof} Using the conclusions of Lemma~\ref{setup}, we can apply Theorem 3.1 of \cite{HeadTail}, Theorem~\ref{mytheorem}, and Theorem~\ref{normformthm} respectively to get that $\vert\beta_{K}'\vert = 1-\chi(\mathbb{G}_{A}') = 1+(t^{-}(D)-1) = t^{-}(D) = s$. Since the colored Jones polynomial and therefore its coefficients are link invariants, then we can conclude that $t^{-}(D)$ and $s$ are link invariants.
\end{proof}

By combining Corollary~\ref{newestsetup}, Corollary~\ref{scor}, and Proposition~\ref{svolbound}, we get the following result, which is a more precise version of Theorem~\ref{sthm} from the introduction.  

\begin{theorem}\label{newsvolbound}
Let $D(K)=\widehat{\beta}$, where $\beta \in B_{3}$ satisfies the assumptions of Lemma~\ref{setup}. Then we have the volume bounds
$$v_{8}\cdot(s-1) \leq \mathrm{vol}(S^{3}\backslash K) < 4v_{8}\cdot s.$$
\end{theorem}

\begin{remark}
Comparing the lower bound on volume of Theorem~\ref{newsvolbound} to that of Proposition~\ref{svolbound}, we get that $v_{8}\cdot (s-1) \geq 4v_{3}\cdot s-276.6$ is equivalent to the condition that $\displaystyle s \leq \frac{276.6-v_{8}}{4v_{3}-v_{8}} \approx 690$. Therefore, the lower bound found in Theorem~\ref{newsvolbound} is sharper than the lower bound provided by Proposition~\ref{svolbound} unless the parameter $s$ from the Schreier normal form is very large. 
\end{remark}

\subsection{The proof of Theorem~\ref{normformthm}}

We will now prove Theorem~\ref{normformthm} using a series of lemmas, beginning by first considering the simpler case that $\beta \in B_{3}$ is a negative braid. 

\begin{lemma} Let $D(K)=\widehat{\beta}$, where $\beta \in B_{3}$ is a negative braid that satisfies the assumptions of Lemma~\ref{setup}. Then $\beta$ is generic and, furthermore, we are able to express the parameters $k$ and $s$ of the Schreier normal form $\beta'$ in terms of the original 3-braid $\beta$ as follows:
\begin{enumerate}
\item[(1)] $k=-\# \left\{\text{induced\ products}\ \sigma_{2}^{n_{2}}\sigma_{1}^{n_{1}}\ \text{of\ negative\ syllables\ of}\ \beta,\ \text{where}\ n_{1}, n_{2} \leq -3 \right\}$.
\item[(2)] $s=t^{-}(D)=\#\left\{\text{negative\ syllables\ in}\ \beta \right\}$.
\end{enumerate}
\end{lemma}

\begin{proof} By Lemma~\ref{setup}, we have that $K$ is hyperbolic. By Proposition~\ref{hypgeneric}, this implies that $\beta$ is generic. Since $\beta$ is a negative braid, then (possibly using cyclic permutation) we may write $\beta$ as $\beta=\sigma_{1}^{n_{1}}\sigma_{2}^{n_{2}}\cdots\sigma_{1}^{n_{2m-1}}\sigma_{2}^{n_{2m}}$, where $n_{i} \leq -3$ and the fact that $\beta$ is nice forces the condition that $m \geq 2$. Applying the Schreier Normal Form Algorithm, we get
\begin{eqnarray*}
\beta & = & \sigma_{1}^{n_{1}}\sigma_{2}^{n_{2}}\cdots\sigma_{1}^{n_{2m-1}}\sigma_{2}^{n_{2m}}\\
\ & = & (xy)^{-n_{1}}(yx)^{-n_{2}}\cdots(xy)^{-n_{2m-1}}(yx)^{-n_{2m}}\\ 
\ & = & (xy)^{-n_{1}-1}(xy^{2})(xy)^{-n_{2}-1}(x^{2})y\cdots(x^{2})y(xy)^{-n_{2m-1}-2}(xy^{2})(xy)^{-n_{2m}-1}x\\ 
\ & = & (xy)^{-n_{1}-1}(xy^{2})(xy)^{-n_{2}-1}(C^{-1})y\cdots(C^{-1})y(xy)^{-n_{2m-1}-2}(xy^{2})(xy)^{-n_{2m}-1}x\\ 
\ & \cong & (C^{-1})^{m-1}(xy)^{-n_{1}-1}(xy^{2})(xy)^{-n_{2}-2}(xy^{2})\cdots(xy^{2})(xy)^{-n_{2m-1}-2}(xy^{2})(xy)^{-n_{2m}-1}x\\ 
\ & \cong & (C^{-1})^{m-1}x(xy)^{-n_{1}-1}(xy^{2})(xy)^{-n_{2}-2}(xy^{2})\cdots(xy^{2})(xy)^{-n_{2m-1}-2}(xy^{2})(xy)^{-n_{2m}-1}\\ 
\ & = & (C^{-1})^{m-1}(x^{2})y(xy)^{-n_{1}-2}(xy^{2})(xy)^{-n_{2}-2}(xy^{2})\cdots(xy^{2})(xy)^{-n_{2m-1}-2}(xy^{2})(xy)^{-n_{2m}-1}\\ 
\ & = & (C^{-1})^{m}y(xy)^{-n_{1}-2}(xy^{2})(xy)^{-n_{2}-2}(xy^{2})\cdots(xy^{2})(xy)^{-n_{2m-1}-2}(xy^{2})(xy)^{-n_{2m}-1}\\ 
\ & \cong & (C^{-1})^{m}(xy)^{-n_{1}-2}(xy^{2})(xy)^{-n_{2}-2}(xy^{2})\cdots(xy^{2})(xy)^{-n_{2m-1}-2}(xy^{2})(xy)^{-n_{2m}-1}y\\ 
\ & = & (C^{-1})^{m}(xy)^{-n_{1}-2}(xy^{2})(xy)^{-n_{2}-2}(xy^{2})\cdots(xy^{2})(xy)^{-n_{2m-1}-2}(xy^{2})(xy)^{-n_{2m}-2}(xy^{2})\\ 
\ & = & (C^{-1})^{m}\sigma_{1}^{n_{1}+2}\sigma_{2}\sigma_{1}^{n_{2}+2}\sigma_{2}\cdots\sigma_{2}\sigma_{1}^{n_{2m-1}+2}\sigma_{2}\sigma_{1}^{n_{2m}+2}\sigma_{2}\\
\ & = & \beta', 
\end{eqnarray*}
\noindent where $\cong$ denotes that cyclic permutation or the fact that $C \in Z(B_{3})$ has been used. Thus, we see that
$$k=-m=-\# \left\{\text{induced\ products}\ \sigma_{2}^{n_{i}}\sigma_{1}^{n_{i+1}}\ \text{of\ negative\ syllables\ of}\ \beta,\ \text{where}\ n_{i}, n_{i+1} \leq -3 \right\}$$
\noindent and
$$s=2m=\#\left\{\text{negative\ syllables\ in}\ \beta \right\}.$$
\noindent Recall that, when looking for the induced products $\sigma_{2}^{n_{i}}\sigma_{1}^{n_{i+1}}$ of negative syllables of $\beta$ to find $k$, we must look cyclically in the braid word. 
\end{proof}

We will now consider the more complicated case that $\beta \in B_{3}$ is not a negative braid. Since $\beta$ is not a negative braid, then (possibly using cyclic permutation) we may write $\beta$ as $\beta=P_{1}N_{1} \cdots P_{t}N_{t}$, where $P_{i}$ denotes a positive syllable of $\beta$ and where $N_{i}$ denotes a maximal length negative induced subword of $\beta$. Note that this decomposition arises as a result of Condition (2a), Condition (2b), and Condition (2c) of Lemma~\ref{setup} (the conditions that positive syllables are cyclically neighbored on both sides by negative syllables).

\bigskip

Our strategy for the proof of this case is to apply the Schreier Normal Form Algorithm to the subwords $P_{i}N_{i}$, showing along the way that cyclic permutation is never needed in applying the algorithm, to show that Theorem~\ref{normformthm} is locally satisfied for the $P_{i}N_{i}$, and to finally show that juxtaposing the subwords $P_{i}N_{i}$ to form $\beta$ allows the local conclusions of Theorem~\ref{normformthm} for the $P_{i}N_{i}$ to combine to give the global conclusion of Theorem~\ref{normformthm} for $\beta$. 

\bigskip 

We now list the types of induced subwords $P_{i}N_{i}$ below. For each subword, we consider the two possible subtypes. Let $p>0$ denote a positive exponent and let the $n_{i}\leq-3$ denote negative exponents.   

\begin{enumerate}
	\item[(1a)] $\sigma_{2}^{p}\sigma_{1}^{n_{1}}$
	\item[(1b)] $\sigma_{1}^{p}\sigma_{2}^{n_{1}}$
	\item[(2a)] $\sigma_{2}^{p}\sigma_{1}^{n_{1}}\sigma_{2}^{n_{2}}\cdots\sigma_{1}^{n_{2m-1}}\sigma_{2}^{n_{2m}}$, where $m \geq 1$
	\item[(2b)] $\sigma_{1}^{p}\sigma_{2}^{n_{1}}\sigma_{1}^{n_{2}}\cdots\sigma_{2}^{n_{2m-1}}\sigma_{1}^{n_{2m}}$, where $m \geq 1$
	\item[(3a)] $\sigma_{2}^{p}\sigma_{1}^{n_{1}}\sigma_{2}^{n_{2}}\cdots\sigma_{1}^{n_{2m-1}}\sigma_{2}^{n_{2m}}\sigma_{1}^{n_{2m+1}}$, where $m \geq 1$
	\item[(3b)] $\sigma_{1}^{p}\sigma_{2}^{n_{1}}\sigma_{1}^{n_{2}}\cdots\sigma_{2}^{n_{2m-1}}\sigma_{1}^{n_{2m}}\sigma_{2}^{n_{2m+1}}$, where $m \geq 1$
\end{enumerate}

\begin{lemma} \label{normlist} Given the list of induced subwords $P_{i}N_{i}$ above, applying the Schreier Normal Form Algorithm produces the following corresponding list of braid words. 
\begin{enumerate}
	\item[(1a)] $\sigma_{2}^{p}\sigma_{1}^{n_{1}}$
	\item[(1b)] $\mathbf{y^{2}}\sigma_{2}^{p-1}\sigma_{1}^{n_{1}}\mathbf{x}$
	\item[(2a)] $(C^{-1})^{m-1}\sigma_{2}^{p}\sigma_{1}^{n_{1}+1}\sigma_{2}\sigma_{1}^{n_{2}+2}\sigma_{2} \cdots \sigma_{2}\sigma_{1}^{n_{2m-1}+2}\sigma_{2}\sigma_{1}^{n_{2m}+1}\mathbf{x}$, where $m \geq 1$
	\item[(2b)] $(C^{-1})^{m}\mathbf{y^{2}}\sigma_{2}^{p-1}\sigma_{1}^{n_{1}+1}\sigma_{2}\sigma_{1}^{n_{2}+2}\sigma_{2} \cdots \sigma_{2}\sigma_{1}^{n_{2m-1}+2}\sigma_{2}\sigma_{1}^{n_{2m}+1}$, where $m \geq 1$
	\item[(3a)] $(C^{-1})^{m}\sigma_{2}^{p}\sigma_{1}^{n_{1}+1}\sigma_{2}\sigma_{1}^{n_{2}+2}\sigma_{2} \cdots \sigma_{2}\sigma_{1}^{n_{2m-1}+2}\sigma_{2}\sigma_{1}^{n_{2m}+2}\sigma_{2}\sigma_{1}^{n_{2m+1}+1}$, where $m \geq 1$
	\item[(3b)] $(C^{-1})^{m}\mathbf{y^{2}}\sigma_{2}^{p-1}\sigma_{1}^{n_{1}+1}\sigma_{2}\sigma_{1}^{n_{2}+2}\sigma_{2} \cdots \sigma_{2}\sigma_{1}^{n_{2m-1}+2}\sigma_{2}\sigma_{1}^{n_{2m}+2}\sigma_{2}\sigma_{1}^{n_{2m+1}+1}\mathbf{x}$, where $m~\geq~1$
\end{enumerate}
Also, when applying the Schreier Normal Form Algorithm to the subwords $P_{i}N_{i}$, cyclic permutation is avoided. Furthermore, letting $k_{i}$ denote the exponent of $C$ in the application of the Schreier Normal Form Algorithm to the subword $P_{i}N_{i}$, we have that a local version of Theorem~\ref{normformthm} holds for the $k_{i}$.
\end{lemma}

\begin{proof} Let $\cong$ denote that the fact that $C \in Z(B_{3})$ has been used. Applying the Schreier Normal Form Algorithm to type (1a), we get 
$$P_{i}N_{i} = \sigma_{2}^{p}\sigma_{1}^{n_{1}} = (xy^{2})^{p}(xy)^{-n_{1}} = \sigma_{2}^{p}\sigma_{1}^{n_{1}}.$$

\noindent Applying the Schreier Normal Form Algorithm to type (1b), we get 
$$P_{i}N_{i} = \sigma_{1}^{p}\sigma_{2}^{n_{1}} = (y^{2}x)^{p}(yx)^{-n_{1}} = y^{2}(xy^{2})^{p-1}(xy)^{-n_{1}}x = y^{2}\sigma_{2}^{p-1}\sigma_{1}^{n_{1}}x$$

\noindent Applying the Schreier Normal Form Algorithm to type (2a), we get  
\begin{eqnarray*}
P_{i}N_{i} & = & \sigma_{2}^{p}\sigma_{1}^{n_{1}}\sigma_{2}^{n_{2}}\cdots\sigma_{1}^{n_{2m-1}}\sigma_{2}^{n_{2m}}\\
\ & = & (xy^{2})^{p}(xy)^{-n_{1}}(yx)^{-n_{2}} \cdots (xy)^{-n_{2m-1}}(yx)^{-n_{2m}}\\
\ & = & (xy^{2})^{p}(xy)^{-n_{1}-1}(xy^{2})(xy)^{-n_{2}-1}(x^{2})y \cdots(x^{2})y(xy)^{-n_{2m-1}-2}(xy^{2})(xy)^{-n_{2m}-1}x\\
\ & = & (xy^{2})^{p}(xy)^{-n_{1}-1}(xy^{2})(xy)^{-n_{2}-1}(C^{-1})y \cdots(C^{-1})y(xy)^{-n_{2m-1}-2}(xy^{2})(xy)^{-n_{2m}-1}x\\
\ & \cong & (C^{-1})^{m-1}(xy^{2})^{p}(xy)^{-n_{1}-1}(xy^{2})(xy)^{-n_{2}-2}(xy^{2}) \cdots\\
\ & \ & \cdots(xy^{2})(xy)^{-n_{2m-1}-2}(xy^{2})(xy)^{-n_{2m}-1}x\\
\ & = & (C^{-1})^{m-1}\sigma_{2}^{p}\sigma_{1}^{n_{1}+1}\sigma_{2}\sigma_{1}^{n_{2}+2}\sigma_{2} \cdots \sigma_{2}\sigma_{1}^{n_{2m-1}+2}\sigma_{2}\sigma_{1}^{n_{2m}+1}x
\end{eqnarray*}

\noindent Applying the Schreier Normal Form Algorithm to type (2b), we get 
\begin{eqnarray*}
P_{i}N_{i} & = & \sigma_{1}^{p}\sigma_{2}^{n_{1}}\sigma_{1}^{n_{2}}\cdots\sigma_{2}^{n_{2m-1}}\sigma_{1}^{n_{2m}}\\
\ & = & (y^{2}x)^{p}(yx)^{-n_{1}}(xy)^{-n_{2}} \cdots (yx)^{-n_{2m-1}}(xy)^{-n_{2m}}\\
\ & = & y^{2}(xy^{2})^{p-1}(xy)^{-n_{1}}(x^{2})y(xy)^{-n_{2}-2}(xy^{2}) \cdots(xy^{2})(xy)^{-n_{2m-1}-1}(x^{2})y(xy)^{-n_{2m}-1}\\
\ & = & y^{2}(xy^{2})^{p-1}(xy)^{-n_{1}}(C^{-1})y(xy)^{-n_{2}-2}(xy^{2}) \cdots(xy^{2})(xy)^{-n_{2m-1}-1}(C^{-1})y(xy)^{-n_{2m}-1}\\
\ & \cong & (C^{-1})^{m}y^{2}(xy^{2})^{p-1}(xy)^{-n_{1}-1}(xy^{2})(xy)^{-n_{2}-2}(xy^{2}) \cdots\\
\ & \ & \cdots(xy^{2})(xy)^{-n_{2m-1}-2}(xy^{2})(xy)^{-n_{2m}-1}\\
\ & = & (C^{-1})^{m}y^{2}\sigma_{2}^{p-1}\sigma_{1}^{n_{1}+1}\sigma_{2}\sigma_{1}^{n_{2}+2}\sigma_{2} \cdots \sigma_{2}\sigma_{1}^{n_{2m-1}+2}\sigma_{2}\sigma_{1}^{n_{2m}+1}
\end{eqnarray*}

\noindent Applying the Schreier Normal Form Algorithm to type (3a), we get 
\begin{eqnarray*}
P_{i}N_{i} & = & \sigma_{2}^{p}\sigma_{1}^{n_{1}}\sigma_{2}^{n_{2}}\cdots\sigma_{1}^{n_{2m-1}}\sigma_{2}^{n_{2m}}\sigma_{1}^{n_{2m+1}}\\
\ & = & (xy^{2})^{p}(xy)^{-n_{1}}(yx)^{-n_{2}} \cdots (xy)^{-n_{2m-1}}(yx)^{-n_{2m}}(xy)^{-n_{2m+1}}\\
\ & = & (xy^{2})^{p}(xy)^{-n_{1}-1}(xy^{2})(xy)^{-n_{2}-1}(x^{2})y \cdots\\
\ & \ & \cdots(x^{2})y(xy)^{-n_{2m-1}-2}(xy^{2})(xy)^{-n_{2m}-1}(x^{2})y(xy)^{-n_{2m+1}-1}\\
\ & = & (xy^{2})^{p}(xy)^{-n_{1}-1}(xy^{2})(xy)^{-n_{2}-1}(C^{-1})y \cdots\\
\ & \ & \cdots(C^{-1})y(xy)^{-n_{2m-1}-2}(xy^{2})(xy)^{-n_{2m}-1}(C^{-1})y(xy)^{-n_{2m+1}-1}\\
\ & \cong & (C^{-1})^{m}(xy^{2})^{p}(xy)^{-n_{1}-1}(xy^{2})(xy)^{-n_{2}-2}(xy^{2}) \cdots\\
\ & \ & \cdots(xy^{2})(xy)^{-n_{2m-1}-2}(xy^{2})(xy)^{-n_{2m}-2}(xy^{2})(xy)^{-n_{2m+1}-1}\\
\ & = & (C^{-1})^{m}\sigma_{2}^{p}\sigma_{1}^{n_{1}+1}\sigma_{2}\sigma_{1}^{n_{2}+2}\sigma_{2} \cdots \sigma_{2}\sigma_{1}^{n_{2m-1}+2}\sigma_{2}\sigma_{1}^{n_{2m}+2}\sigma_{2}\sigma_{1}^{n_{2m+1}+1}
\end{eqnarray*}

\noindent Applying the Schreier Normal Form Algorithm to type (3b), we get 
\begin{eqnarray*}
P_{i}N_{i} & = & \sigma_{1}^{p}\sigma_{2}^{n_{1}}\sigma_{1}^{n_{2}}\cdots\sigma_{2}^{n_{2m-1}}\sigma_{1}^{n_{2m}}\sigma_{2}^{n_{2m+1}}\\
\ & = & (y^{2}x)^{p}(yx)^{-n_{1}}(xy)^{-n_{2}} \cdots (yx)^{-n_{2m-1}}(xy)^{-n_{2m}}(yx)^{-n_{2m+1}}\\
\ & = & y^{2}(xy^{2})^{p-1}(xy)^{-n_{1}}(x^{2})y(xy)^{-n_{2}-2}(xy^{2}) \cdots\\
\ & \ & \cdots(xy^{2})(xy)^{-n_{2m-1}-1}(x^{2})y(xy)^{-n_{2m}-2}(xy^{2})(xy)^{-n_{2m+1}-1}x\\
\ & = & y^{2}(xy^{2})^{p-1}(xy)^{-n_{1}}(C^{-1})y(xy)^{-n_{2}-2}(xy^{2}) \cdots\\
\ & \ & \cdots(xy^{2})(xy)^{-n_{2m-1}-1}(C^{-1})y(xy)^{-n_{2m}-2}(xy^{2})(xy)^{-n_{2m+1}-1}x\\
\ & \cong & (C^{-1})^{m}y^{2}(xy^{2})^{p-1}(xy)^{-n_{1}-1}(xy^{2})(xy)^{-n_{2}-2}(xy^{2}) \cdots\\
\ & \ & \cdots(xy^{2})(xy)^{-n_{2m-1}-2}(xy^{2})(xy)^{-n_{2m}-2}(xy^{2})(xy)^{-n_{2m+1}-1}x\\
\ & = & (C^{-1})^{m}y^{2}\sigma_{2}^{p-1}\sigma_{1}^{n_{1}+1}\sigma_{2}\sigma_{1}^{n_{2}+2}\sigma_{2} \cdots \sigma_{2}\sigma_{1}^{n_{2m-1}+2}\sigma_{2}\sigma_{1}^{n_{2m}+2}\sigma_{2}\sigma_{1}^{n_{2m+1}+1}x
\end{eqnarray*}

By inspecting the cases above, it can be seen that cyclic permutation is never used. Let $k_{i}$ denote the exponent of $C$ in the application of the Schreier Normal Form Algorithm to the subword $P_{i}N_{i}$. In each case above, it can be see that
$$k_{i}=-\# \left\{\text{induced\ products}\ \sigma_{2}^{n_{j,i}}\sigma_{1}^{n_{j+1,i}}\ \text{of\ negative\ syllables\ of}\ P_{i}N_{i},\ \text{where}\ n_{j,i}, n_{j+1,i} \leq -3 \right\},$$
which is a local version of conclusion (1) from Theorem~\ref{normformthm}. 

\end{proof}

The fact that cyclic permutation is never used in applying the Schreier Normal Form Algorithm to the subwords $P_{i}N_{i}$ is very important because the goal is to juxtapose the resulting braid words from Lemma~\ref{normlist} and claim that this gives, after some minor modifications, the normal form of the full braid word $\beta$. We are now ready to combine the above results to prove Theorem~\ref{normformthm} for the case that $\beta \in B_{3}$ is not a negative braid. The following lemma establishes that $\beta$ is generic and establishes the result for the parameter $k$ from the Schreier normal form (conclusion (1) from Theorem~\ref{normformthm}).  

\begin{lemma} \label{klemma} Let $D(K)=\widehat{\beta}$, where $\beta \in B_{3}$ is a nonnegative braid that satisfies the assumptions of Lemma~\ref{setup}. Then $\beta$ is generic and the parameter $k$ from the Schreier normal form $\beta'$ can be expressed in terms of the original 3-braid $\beta$ as 
$$k=-\# \left\{\text{induced\ products}\ \sigma_{2}^{n_{2}}\sigma_{1}^{n_{1}}\ \text{of\ negative\ syllables\ of}\ \beta,\ \text{where}\ n_{1}, n_{2} \leq -3 \right\}.$$
\end{lemma}

\begin{proof} By Lemma~\ref{setup}, we have that $K$ is hyperbolic. By Proposition~\ref{hypgeneric}, this implies that $\beta$ is generic. Let us now consider which types of induced subwords $P_{i+1}N_{i+1}$ can (cyclically) follow a given induced subword $P_{i}N_{i}$. Since $\beta$ is assumed to be cyclically reduced into syllables and since the the induced subwords $P_{i}N_{i}$ contain unbroken syllables of $\beta$, then (looking at the list of induced subwords $P_{i}N_{i}$ preceding Lemma~\ref{normlist}) we get that
\begin{itemize}
	\item the subwords $P_{i}N_{i}$ of type (1) and type (3) with subtype either (a) or (b) can be followed by subwords $P_{i+1}N_{i+1}$ of \textbf{the same} subtype (a) or (b), respectively. 
\item the subwords $P_{i}N_{i}$ of type (2) with subtype either (a) or (b) can be followed by subwords $P_{i+1}N_{i+1}$ of \textbf{different} subtype (b) or (a), respectively. 
\end{itemize}

It is important to note that the rules for juxtaposition given above also apply after the Schreier Normal Form Algorithm is applied to the list of induced subwords $P_{i}N_{i}$ preceding Lemma~\ref{normlist}. Abusing terminology slightly, call the braids words in the list given in Lemma~\ref{normlist} the \emph{normal forms of the $P_{i}N_{i}$}. 

Let $k_{i}$ denote the exponent of $C$ in the normal form of $P_{i}N_{i}$. Recall that $C \in Z(B_{3})$ commutes with the generators of $B_{3}$. Thus, when juxtaposing the normal form of $P_{i}N_{i}$ with the normal form of $P_{i+1}N_{i+1}$, we may move the factor $C^{k_{i+1}}$ out of the way, moving it from the beginning of the normal form of $P_{i+1}N_{i+1}$ to the beginning of the normal form of $P_{i}N_{i}$. This fact will be utilized below. 

Looking at the list of braids words given in Lemma~\ref{normlist}, notice that half of the braid words may potentially contain the variable $\mathbf{x}$ at the end of the word and half of the braid words may contain the expression $\mathbf{y^{2}}\sigma_{2}^{p-1}$ at the beginning of the word (immediately after the $C^{k_{i+1}}$ term that will be moved out of the way). We can now see that juxtaposing the normal form of $P_{i}N_{i}$ with that of $P_{i+1}N_{i+1}$ either
\begin{enumerate}
	\item[(1)] involves neither $\mathbf{x}$ at the end of the normal form of $P_{i}N_{i}$ nor $\mathbf{y^{2}}\sigma_{2}^{p-1}$ at the beginning of the normal form of $P_{i+1}N_{i+1}$, or
	\item[(2)] involves (after moving $C^{k_{i+1}}$ out of the way) both an $\mathbf{x}$ at the end of the normal form of $P_{i}N_{i}$ and $\mathbf{y^{2}}\sigma_{2}^{p-1}$ at the beginning of the normal form of $P_{i+1}N_{i+1}$.
\end{enumerate}
\noindent Consequently, upon juxtaposing the normal forms of all of the induced subwords together, we see that all factors $\mathbf{x}$ and $\mathbf{y^{2}}\sigma_{2}^{p-1}$ in the list of normal forms combine to form $$\mathbf{xy^{2}}\sigma_{2}^{p-1}=\boldsymbol{\sigma_{2}}\sigma_{2}^{p-1}=\sigma_{2}^{p}.$$ 
Note, in particular, that juxtaposing the normal form of $P_{i}N_{i}$ with that of $P_{i+1}N_{i+1}$ does \textbf{not} create any new nontrivial powers of $C$.  Therefore, since $C$ commutes with the generators of $B_{3}$, then we may group all $C^{k_{i}}$ terms together at the beginning of the normal form braid word. Furthermore, by applying the conclusions of Lemma~\ref{normlist}, we have that 
$$k_{i}=-\# \left\{\text{induced\ products}\ \sigma_{2}^{n_{j,i}}\sigma_{1}^{n_{j+1,i}}\ \text{of\ negative\ syllables\ of}\ P_{i}N_{i},\ \text{where}\ n_{j,i}, n_{j+1,i} \leq -3 \right\}.$$
We also have that juxtaposing $P_{i}N_{i}$ with $P_{i+1}N_{i+1}$ does \textbf{not} create any new induced products $\sigma_{2}^{n_{j}}\sigma_{1}^{n_{j+1}}$ of negative syllables. This is because juxtaposing $P_{i}N_{i}$ with $P_{i+1}N_{i+1}$ only joins a negative syllable with a positive syllable. With this information, we are now able to conclude that
\begin{eqnarray*}
k & = & \sum_{i=1}^{t} k_{i}\\
\ & = & \sum_{i=1}^{t} -\# \left\{\text{induced\ products}\ \sigma_{2}^{n_{j,i}}\sigma_{1}^{n_{j+1,i}}\ \text{of\ negative\ syllables\ of}\ P_{i}N_{i},\ \text{where}\ n_{j,i}, n_{j+1,i} \leq -3 \right\}\\
\ & = & -\# \left\{\text{induced\ products}\ \sigma_{2}^{n_{j}}\sigma_{1}^{n_{j+1}}\ \text{of\ negative\ syllables\ of}\ \beta,\ \text{where}\ n_{j}, n_{j+1} \leq -3 \right\}.
\end{eqnarray*}          

\end{proof}

The following lemma proves conclusion (2) from Theorem~\ref{normformthm}, the result concerning the parameter $s$ from the Schreier normal form, for the case that $\beta \in B_{3}$ is not a negative braid. 

\begin{lemma} Let $D(K)=\widehat{\beta}$, where $\beta \in B_{3}$ is a nonnegative braid that satisfies the assumptions of Lemma~\ref{setup}. Then the parameter $s$ from the Schreier normal form $\beta'$ can be expressed in terms of the original 3-braid $\beta$ as 
$$s=t^{-}(D)=\#\left\{\text{negative\ syllables\ in}\ \beta \right\}.$$
\end{lemma}

\begin{proof}
We need to relate the global parameter $s$ from the Schreier normal form to local versions of the parameter $s$. Let $s_{i}$ denote the local version of the parameter $s$, which comes from the normal form of the subword $P_{i}N_{i}$ and will be more precisely defined below.  

As seen in the proof of Lemma~\ref{klemma} above, juxtaposing the normal forms of the subwords $P_{i}N_{i}$ to create a braid word groups together the $\mathbf{x}$ and $\mathbf{y^{2}}$ factors in the normal forms in such a way that they are absorbed into $\sigma_{2}^{p-1}$ to form $\sigma_{2}^{p}$. Also recall that we can collect together all individual powers of $C$ from each $P_{i}N_{i}$ normal form and use commutativity to form a single power of $C$ at the beginning of the normal form braid word. Thus, after juxtaposition of the normal forms of the $P_{i}N_{i}$, what results is a braid word that looks like $C^{k}W_{1} \cdots W_{t}$, where $k \in \mathbb{Z}$ and $W_{i}$ is an alternating word that is positive in $\sigma_{2}$, negative in $\sigma_{1}$, begins with a $\sigma_{2}$ syllable, and ends with a $\sigma_{1}$ syllable. To see this, recall the list of $P_{i}N_{i}$ normal forms given in Lemma~\ref{normlist}. 

Given an alternating subword $W_{i}=\sigma_{2}^{p_{1,i}}\sigma_{1}^{n_{1,i}}\cdots\sigma_{2}^{p_{q,i}}\sigma_{1}^{n_{q,i}}$ as described above, we define $s_{i}=q$. To provide an example, for the normal form $C^{-1}\sigma_{2}^{p}\sigma_{1}^{n_{1}+1}\sigma_{2}\sigma_{1}^{n_{2}+2}\sigma_{2}\sigma_{1}^{n_{3}+2}\sigma_{2}\sigma_{1}^{n_{4}+1}\mathbf{x}$ of Type (2a), we have that $s_{i}=4$. Consider the product
$$W_{i}W_{i+1}=(\sigma_{2}^{p_{1,i}}[\sigma_{1}^{n_{1,i}}\cdots\sigma_{2}^{p_{q,i}}\sigma_{1}^{n_{q,i}})\cdot(\sigma_{2}^{p_{1,i+1}}]\sigma_{1}^{n_{1,i+1}}\cdots\sigma_{2}^{p_{r,i+1}}\sigma_{1}^{n_{r,i+1}}).$$ Note that the subword $[\sigma_{1}^{n_{1,i}}\cdots\sigma_{2}^{p_{q,i}}\sigma_{1}^{n_{q,i}}\cdot\sigma_{2}^{p_{1,i+1}}]$ looks like the alternating part of a generic braid word, the part of the normal form for which the parameter $s$ measures the length. From this perspective of (cyclically) borrowing the first syllable of the next subword, the local parameter $s_{i}$ makes sense as being the local version of the global parameter $s$. 

Recall that, since $\beta$ is cyclically reduced into syllables, then the number of negative twist regions in $D(K)=\widehat{\beta}$ corresponds to the number of negative syllables in $\beta$. Also, note that the decomposition of $\beta$ into induced subwords $P_{i}N_{i}$ (with unbroken syllables) gives that the number of negative syllables in $\beta$ is the sum of the numbers of negative syllables in the $P_{i}N_{i}$. Furthermore, returning to Lemma~\ref{normlist} and its proof, it can be seen that the number of negative syllables in the subword $P_{i}N_{i}$ is equal to the local parameter $s_{i}$ from the normal form of $P_{i}N_{i}$. Finally, since juxtaposing the normal forms of the $P_{i}N_{i}$ gives a braid word $C^{k}W_{1} \cdots W_{t}$ that (by cyclic permutation and the commutativity of $C$) is equivalent to the generic normal form of $\beta$, then the parameters $s_{i}$ from the $W_{i}$ sum to give the parameter $s$ from the Schreier normal form of $\beta$. With this information, we are now able to conclude that
\begin{eqnarray*}
t^{-}(D) & = & \#\left\{\text{negative\ syllables\ in\ }\beta\right\}\\
\ & = & \sum_{i=1}^{t} \#\left\{\text{negative\ syllables\ in\ }P_{i}N_{i}\right\}\\
\ & = & \sum_{i=1}^{t} s_{i}\\
\ & = & s.
\end{eqnarray*}  

\end{proof}

\bibliography{mybib}
\bibliographystyle{plain}

\end{document}